\documentclass[12pt]{amsart}

\usepackage{amsmath,amssymb,amsfonts}

\newtheorem{theorem}{Theorem}[section]
\newtheorem{lemma}[theorem]{Lemma}
\newtheorem{proposition}[theorem]{Proposition}

\newtheorem{corollary}[theorem]{Corollary}
\newtheorem{thm}{Theorem}

\newtheorem{cor}{Corollary}

\begin{document}

\title[T-braces]{On the structure of braces whose subideals are ideals} 

\author{Martyn R. Dixon}
\address[Martyn Dixon]
{Professor Emeritus, Department of Mathematics\\
University of Alabama\\
Tuscaloosa, AL 35487-0350, U.S.A.}
\email{mdixon@retiree.ua.edu}
\author{Leonid A. Kurdachenko}
\address[Leonid Kurdachenko]
{Department of Algebra, Facultet of mathematic and mechanik\\
National University of Dnepropetrovsk\\
Gagarin prospect 72\\
Dnipro 10, 49010, Ukraine.}
\email{lkurdachenko@i.ua, lkurdachenko@gmail.com}
\author[I. Subbotin]{Igor Ya. Subbotin}
\address[Igor Ya. Subbotin]{Department of Mathematics and Natural Sciences, National University\\
5245 Pacific Concourse Drive,  \\
Los Angeles, CA 90045-6904, USA.}
\email{isubboti@nu.edu}

\begin{abstract} 
This article begins the study of T-braces, those skew left braces of abelian type in which the relation of being an ideal is a transitive relation.
\end{abstract} 
\thanks{The second author is grateful for the support of the   Isaac Newton Institute for Mathematical Sciences, Cambridge and the University of Edinburgh provided in the framework of the LMS Solidarity
Supplementary Grant Programme. This work has been also supported by the EPSRC Grant  No. EP/R014604/1. 
He is sincerely grateful to Agata Smoktunowicz. 
}
\dedicatory{To our friends Patrizia Longobardi and Mercede Maj on the occasion of their retirement}



\maketitle

\newcommand{\sub}[2]{\langle#1,#2\rangle}
\newcommand{\cyclic}[1]{\langle #1\rangle}
\newcommand{\norm}{\triangleleft\,}
\newcommand{\core}[2]{\text{core}\,_{#1}\,#2}
\newcommand{\cpin}{C_{p^{\infty}}}
\newcommand{\co}[1]{\text{$\mathbf{#1}$}}
\newcommand{\cogo}[2]{\co{#1}\mathfrak{#2}}
\newcommand{\ann}{\text{Ann\,}}
\newcommand{\spec}{\text{Spec\,}}
\newcommand{\listel}[2]{#1_1,\dots, #1_{#2}}
\newcommand{\aut}{\text{Aut}\,}
\newcommand{\cent}[2]{C_{#1}(#2)}
\newcommand{\Dr}{\text{Dr}\,}
\newcommand{\dir}[3]{\underset{#1\in #2}{\Dr}#3_{#1}}
\newcommand{\nocl}[2]{\langle #1\rangle^{#2}}
\newcommand{\qn}{\,\text{qn}\,}
\newcommand{\comm}[2]{#1^{-1}#2^{-1}#1#2}
\newcommand{\conj}[2]{#1^{-1}#2#1}
\newcommand{\izer}[2]{N_{#1}(#2)}
\newcommand{\gl}{\text{GL\,}}
\newcommand{\subgp}[3]{\langle#1_{#2}\mid #2\in#3\rangle}
\newcommand{\direct}[3]{\underset{#1\in #2}{\Dr}\langle#3_{#1}\rangle}
\newcommand{\omn}[1]{\Omega_n{(#1)}}
\newcommand{\Hom}{\text{Hom}\,}
\newcommand{\Cr}[3]{\underset{#1\in #2}{\text{Cr}}\,#3_{#1}}
\newcommand{\inflist}[1]{\{#1_1,#1_2,\dots\}} 
\newcommand{\cart}{\text{Cr}\,} 
\newcommand{\carti}[3]{\underset{#1\geq #2}\cart #3_{#1}} 
\newcommand{\fgsub}[2]{\langle #1_1,\dots,#1_{#2}\rangle}
\newcommand{\cwr}{\,\bar{\wr}\,}
\newcommand{\minn}{\infty$-$\overline{\mathfrak{N}}}
\newcommand{\mnn}{\overline{\mathfrak{N}}}
\newcommand{\mnp}{\overline{\mathcal{P}}}
\newcommand{\mns}{\overline{\mathfrak{S}}}
\newcommand{\minsd}{\infty$-$\overline{\mathfrak{S}_d}}
\newcommand{\minnc}{\infty$-$\overline{\mathfrak{N}_c}}
\newcommand{\ms}{\mathfrak{S}}
\newcommand{\mn}{\mathfrak{N}}
\newcommand{\mf}[1]{\mathfrak{#1}}
\newcommand{\sdr}{\mathfrak{S}_d(r)}
\newcommand{\sd}{\mathfrak{S}_d}
\newcommand{\mfr}{\mathfrak{R}}
\newcommand{\ncr}{\mathfrak{N}_c(r)}
\newcommand{\mfsr}{\overline{\ms\mfr}}
\newcommand{\lslf}{\overline{(\cogo{L}{S})(\cogo{L}{F})}}
\newcommand{\mcp}{\mathcal{P}}
\newcommand{\mcps}{\mathcal{P}^*}
\newcommand{\mfas}{\mathfrak{A}^*}
\newcommand{\mfa}{\mathfrak{A}}
\newcommand{\mfss}{\mathfrak{S}^*}
\newcommand{\mfns}{\mathfrak{N}^*}
\newcommand{\wmcnq}{max-$\infty$-$\overline{\text{qn}}$ }
\newcommand{\om}{\omega}
\newcommand{\bog}[1]{\bar{\omega}(#1)}
\newcommand{\og}[1]{\omega(#1)}
\newcommand{\wi}{\overline{\om}_i}
\newcommand{\w}{\overline{\om}}
\newcommand{\fsn}{\emph{f}-subnormal }
\newcommand{\wiw}{\om_i}
\newcommand{\ww}{\om}

\newcommand{\zefin}[1]{\zeta_{\infty}(\star,#1)}
\newcommand{\zet}[1]{\zeta(\star,#1)}
\newcommand{\zen}[2]{\zeta_{#1}(\star,#2)}
\newcommand{\al}{\alpha}
\newcommand{\ga}{\gamma}
\newcommand{\be}{\beta}
\newcommand{\si}{\sigma}
\newcommand{\la}{\lambda}
\newcommand{\de}{\delta}
\newcommand{\La}{\Lambda}
\newcommand{\Om}{\Omega}
\newcommand{\ze}{\zeta}
\newcommand{\soc}[1]{\text{Soc}(#1)}
\newcommand{\zl}[1]{\text{zl}(#1)}
\newcommand{\ldl}{\leq\dots\leq}
\newcommand{\br}[1]{\textbf{br}(#1)}
\newcommand{\oq}[1]{(#1_1,#1_2,#1_3,#1_4)}
\newcommand{\su}[2]{#1_{#2}}
\newcommand{\lev}[2]{{\textbf{#1-Lev}}(#2)}
\newcommand{\cyclicp}[1]{\cyclic{#1}_+}
\newcommand{\omegn}[2]{\Om_{#1}(#2)}

\section{Introduction} \label{s:intro}

A \emph{left brace  (or skew left brace of abelian type)}  is a set $A$ together with two binary operations, addition denoted by $+$ and multiplication denoted by $\cdot$ (which is often omitted), satisfying the following conditions:

\begin{enumerate}
\item[LB1] $A$ is an abelian group under addition;
\item[LB2] $A$ is a group under multiplication;
\item[LB3] $a(b+c)= ab+ac -a$ for all $a, b, c\in  A$.
\end{enumerate}

In this paper, all braces will be left braces in this sense.
Sometimes we shall denote the additive group of $A$ by $(A,+)$ and the multiplicative group of $A$ by $(A,\cdot)$.  
Braces were first introduced by W. Rump in \cite{wR05} as a generalization of Jacobson radical rings in order to help study involutive set-theoretic solutions of the Yang-Baxter equation.  

Braces can be viewed as an algebraic structure
representing a coalescence of two groups. Unlike with rings, in braces the additive and mulitplicative identities coincide 
 and we denote this common neutral element by $0$ or $1$, as appropriate. 
 
 If $A$ is a left brace, then a subset $S$ of $A$ is called a \emph{subbrace} (more precisely a \emph{left subbrace}) 
 if $S$ is closed under  addition and multiplication and is a left brace by 
restriction of these operations to  $S$.  Thus a subset $S$ of $A$ is a subbrace of $A$ if and only if 
$(S,+)\leq (A,+)$ and $(S,\cdot)\leq (A,\cdot)$. We shall denote that $S$ is a subbrace of $A$ by $S\leq A$.

If $M$ is a subset of a left brace $A$ then the \emph{subbrace $B=\br{M}$ of $A$ generated by $M$} is the intersection of all subbraces $C$ of $A$ that contain $M$. 
 If $M$ is a finite set, then $B$ will be called \emph{finitely generated}. If $M=\{a\}$ then we write $\br{a}$ rather than $\br{\{a\}}$ and call $\br{a}$ a one-generator brace.  We refer the reader to \cite {CFT23,JVV22} for further details and properties.

 For the left brace $A$ and for each element $a\in A$ we define a map $\lambda_a:A\longrightarrow A$ by $\lambda_a(x)=ax-a$ for all $x\in A$. The map $\lambda_a$ is an automorphism of $(A,+)$. Furthermore we define $a\star b =ab-a-b$ for all $a,b\in A$. It is easy to see that $a\star b=\lambda_a(b)-b$. 
 
 A left brace $A$ is called \emph{trivial} or \emph{abelian} if $a\star b=0$ for all $a,b\in A$, or equivalently, $a+b=ab$.  In this case 
the addition and multiplication coincide.
 
We recall that a subbrace $L$ of the left brace $A$ is called a \emph{left ideal} of $A$ if $a\star b\in L$ for all elements $a\in A$ and $b\in L$.  
This means that if $x, y\in L$, then $x-y\in L$ and if $a\in A$ and $z\in L$, then $\la_a(z)\in L$,
as  in \cite{fC18}.
Furthermore, a subbrace $L$ of a left brace $A$ is an \emph{ideal} if $a\star z, z\star a\in L$ for all $a\in A, z\in L$.  We can then form a new brace, the quotient brace $A/L$, in a natural way. Quite often, it happens that we are interested in studying braces where the quotient braces are abelian, so that the operations of addition and multiplication coincide in the quotient.
 
 If $A$ is a left brace and $K,L$ are subbraces, then we let $K\star L$ denote the subgroup of $(A,+)$ generated by the elements $x\star y$, where $x\in K, y\in L$.
 
 We define two  canonical series of the left brace $A$.  First, let $A^{(1)}=A$ and recursively define $A^{(\alpha +1)}=A^{(\alpha)}\star A$ for all ordinals $\alpha$ and as usual set $A^{(\lambda)}=\cap_{\mu<\lambda}A^{(\mu)}$ for all limit ordinals $\lambda$.   It is shown in \cite{fC18} that $A^{(n)}$ is an ideal of $A$. This has been extended in \cite{DKS25} to include all ordinals $\alpha$. Clearly, $A^{(\alpha +1)}\leq A^{(\alpha)}$, for all ordinals $\alpha$.


 Similarly, let $A^{1}=A$, $A^{\alpha+1}=A\star A^{\alpha}$, for all ordinals $\alpha$ and $A^{\lambda}=\cap_{\mu<\lambda}A^{\mu}$ for all limit ordinals $\lambda$. It is shown in \cite{fC18} that $A^n$ is always a left ideal of $A$ a fact which has been extended in \cite{DKS25} to include all ordinals $\alpha$. Clearly, $A^{\alpha +1}\leq A^{\alpha}$, for all ordinals $\alpha$.


 
We shall say that $A$ is \emph{Smoktunowicz-nilpotent} if there are natural numbers $n,k$ such that $A^{(n)}=A^{k}$. Such braces were introduced by A. Smoktunowicz in the paper \cite{aS18} and are called 
\emph{strongly nilpotent} in \cite{aS22}.
 
It is clear that if $A$ is a subbrace of  $B$ and $B$ is a subbrace of $C$, then $A$ is a subbrace of $C$. However, as with many other algebraic structures, there is in general no such transitive relation for ideals. Groups in which normality is a transitive relation have been studied by several authors ( see \cite{wG57} and \cite{dR64}, for example) and a brief overview of the research on such groups and related topics can be found in the article \cite{KLMO23}. Similar studies have been conducted in other algebraic structures; for example those Leibniz algebras in which the property of being an ideal is a transitive relation were studied in the paper \cite{KSY18}. In group theory, groups in which normality is a transitive relation are called T-groups and, by analogy, a left brace $A$ is called a \emph{T-brace}, if the property of being an ideal is a transitive relation on $A$. 

A T-group is a generalization of a Dedekind group, a group in which every subgroup is normal. Dedekind braces were studied in the paper \cite{BEKP25} and in this paper we begin the study of their generalization, T-braces. The class of T-groups which satisfy other natural restrictions have been studied further. One such natural restriction is nilpotency--nilpotent T-groups are clearly Dedekind groups. 
In this paper our approach to nilpotency is based on the following concept, first introduced by Bonatto and Jedli\v{c}ka~\cite{BJ21} 
 and Jespers, Van Antwerpen and Vendramin~\cite{JVV22} 
 
Let $A$ be a left brace. The set 
 \begin{align*}
 \zeta(\star, A)&=\{a\in A| a\star x=x\star a= 0 \text{ for all }x\in A\}\\
 &=\{a\in A| ax=a+x=xa \text{ for all } x\in A\}\\
 &=\{a\in A| \la_a(x)=x \text{ and } \la_x(a)=a \text{ for all }x\in A\}
 \end{align*}
 is here called the \emph{$\star$-center} of $A$.  The set $\zet{A}$ is an ideal of $A$ which is contained in the center, $\zeta(A)$, of $(A,\cdot)$; this known result \ appears in \cite{BJ21} and \cite{JVV22}. We denote the center of the multiplicative group of $A$ by $\zeta(A)$.

Starting from the $\star$-center and by  analogy with groups we construct the \emph{upper $\star$-central series}
 \[
 0=\zeta_0(\star, A)\leq \zeta_1(\star, A)\leq \dots \zeta_{\alpha}(\star, A)\leq \zeta_{\alpha+1}(\star, A) \leq 
\dots \zeta_{\gamma}(\star, A)
\]
 as follows. We let 
 \begin{align*}
 \zeta_{1}(\star, A)&=\zeta(\star, A) \text{ and }\\
 \zeta_{\alpha+1}(\star, A)/\zeta_{\alpha}(\star, A)&=\zeta(\star, A/\zeta_{\alpha}(\star, A))
 \end{align*}
  for all ordinals $\alpha$; as usual for limit ordinals $\lambda$ we set $\zeta_{\lambda}(\star, A)=\cup_{\mu<\lambda}\zeta_{\mu}(\star, A)$.  

Each term of this series is an ideal of $A$. The last term $\zeta_{\infty}(\star, A)=\zeta_{\gamma}(\star, A)$ of this series is called the  \emph{upper $\star$-hypercenter} of $A$ and we denote the length of the upper $\star$-central series of $A$ by $\text{zl}(A)$. Furthermore, if $A=\zeta_{\infty}(\star, A)$, then $A$ is called a \emph{$\star$-hypercentral} brace; in this case, if $\text{zl}(A)$ is finite, then we say that $A$ is \emph{$\star$-nilpotent}.  This series has also been discussed in \cite{BJ21, CFT23} and \cite{mT23} but there this series is called the \emph{upper annihilator series} and the last term is then called the \emph{hyper-annihilator}; a $\star$-hypercentral brace is called  \emph{annihilator hypercentral} in \cite{mT23} and, in \cite{BJ21, JVV22},  $\star$-nilpotency is called \emph{central nilpotency (and respectively) annihilator nilpotency}. That Smoktunowicz nilpotency and $\star$-nilpotency are equivalent is evident, for example, in \cite[Proposition 2.10]{DKS25}.
Further work concerning $\star$-hypercentral braces can be found in \cite{DKS25}.

In this paper we begin the study of Smoktunowicz-nilpotent T-braces. The situation here is more complicated than with  groups and Leibniz algebras. Our first results are devoted to the study of left braces in which the additive group is not periodic. Our first main result is as follows.

\begin{thm}\label{thma}
Let $A$ be a T-brace. If the additive group of the $\star$-center of $A$ is torsion-free, then the upper $\star$-hypercenter of $A$ and the $\star$-center of $A$ coincide. 
\end{thm}

This result has the following easy consequences
\begin{cor}\label{cor1}
Let $A$ be a T-brace whose additive group is torsion-free. If $A$ is $\star$-hypercentral, then $A$ is abelian.
\end{cor}

\begin{cor}\label{cor2}
Let $A$ be a T-brace whose additive group is torsion-free. If $A$ is Smoktunowicz-nilpotent, then $A$ is abelian.
\end{cor}

We shall say that a $\star$-hypercentral brace is periodic if its additive group $(A,+)$ is periodic and non-periodic otherwise. 

\begin{cor}\label{cor3}
Let $A$ be a non-periodic $\star$-hypercentral T-brace and let $T$ be the torsion subgroup of the additive group of $A$. Then $T$ is an ideal of $A$ and $A/T$ is an abelian brace.
\end{cor}

We shall also prove the following result.

\begin{thm}\label{thmb}
Let $A$ be a T-brace. Suppose that $A$ is Smoktunowicz-nilpotent and that $\text{zl}(A)=k$ for some positive integer $k$. If the additive group of $A$ is not periodic, then the Sylow $p$-subgroup of $K=A\star A$ is finite of order at most $p^{k-1}$ for every prime $p$.
\end{thm}

The layout of our paper is as follows. In Section~\ref{s:smok} we give some results that we shall require later and also give a brief overview of some results connected with $\star$-nilpotent and $\star$-hypercentral braces. In Section~\ref{s:torfree} we state and prove the results which enable us to prove Theorem~\ref{thma} and in Section~\ref{s:nonper} we prove   pertinent results for Theorem~\ref{thmb}.

 \section{Preliminary results and an overview of $\star$-hypercentral braces}\label{s:smok}
 
 We shall need some of the following properties of $\star$ and $\lambda_a$ whose proofs can be found in \cite{fC18} or \cite{JVV22}.

\begin{proposition}\label{p:p1}
Let $A$ be a left brace.  Then
\begin{enumerate}
\item[  (i)] $ab^{-1}=-\la_{ab^{-1}}(b) +a, a-b=\la_b(b^{-1})a, a+b=a \la_{a^{-1}}(b)$;
\item[ (ii)] $a\star (b+c)=a\star b+ a\star c$;
\item[(iii)] $(ab)\star c=a\star (b\star c)+b\star c+ a\star c$;
\item[ (iv)] $(a+b)\star c=a\star (\lambda_{a^{-1}}(b)\star c) +(\lambda_{a^{-1}}(b)\star c)+a\star c$;
\item[  (v)] $\la_a(x+y)=\la_a(x)+\la_a(y), (\la_a\circ \la_b)(x)=\la_{ab}(x)$;
\item[ (vi)] $\lambda_y(b\star a)=yby^{-1}\star \lambda_y(a)$;
\item[(v)] $yby^{-1}=\lambda_y(\lambda_b(y^{-1}) -y^{-1} +b)=\lambda_y(b\star y^{-1} +b)$
\end{enumerate}
for all elements $a,b,c,x,y\in A$.
\end{proposition}

We note,  in particular, that Proposition~\ref{p:p1}(iii) shows that if $c\in\zen{2}{A}$, then the map $\be_c: (A,\cdot)\longrightarrow (\zet{A},+)$ defined by $\be_c(x)=x\star c$ is a group homomorphism.

We shall also require the following result, which is well-known and its proof is omitted.

\begin{lemma}\label{l:l1}
Let $n,k$ be natural numbers and let $a, b_1,b_2,\dots, b_n, c_1,c_2,\dots, c_k$ be elements of the  left brace $A$.  Then 
\[
a(b_1+\dots+ b_n-c_1-\dots - c_k)=ab_1+ab_2+\dots +ab_n-ac_1-\dots -ac_k +(k-n+1)a.
\]
\end{lemma}

The following result is also quite easy.

\begin{lemma}\label{l:l3}
Let $A$ be a left brace and let $k$ be an integer. Then 
\begin{enumerate}
\item[ (i)] $x\star ka=k(x\star a)$ for all $a,x\in A$;
\item[(ii)] If $a\in\zen{2}{A}$, then $(ka)\star x=k(a\star x)$ for all $x\in A$. 
\end{enumerate}
\end{lemma}

\begin{proof}
 (i) is immediate from Proposition~\ref{p:p1}. 

(ii) Let $k$ be a positive integer and let $Z=\zet{A}$. Since $\zen{2}{A}/Z$ is abelian it follows that $ka+Z=a^k+Z=a^kZ$ so that $ka=a^ku$ for some element $u\in Z$. Since $u\star x=0$ and $a\star x\in Z$, Proposition~\ref{p:p1}(iii) implies that 
\[
ka\star x=a^k\star x= k(a\star x),
\]
using a further induction argument. Also
\[
0=0\star x=1\star x=aa^{-1}\star x=a\star (a^{-1}\star x)+a\star x+ a^{-1}\star x= a\star x+a^{-1}\star x,
\]
upon again using Proposition~\ref{p:p1} and the fact that $a^{-1}\in \zen{2}{A}$. Thus $a^{-1}\star x=-(a\star x)$ and a further easy induction completes the proof.
\end{proof}

The next result  is also probably well-known but we give a proof for the sake of completeness. We shall throughout let $\cyclic{a}$ denote the additive cyclic subgroup generated by the element $a$ in the brace $A$, although often we shall explicitly state our meaning again.

\begin{lemma}\label{l:extra1}
Let $A$ be a left brace and let $a\in A$.
\begin{enumerate}
\item[  (i)] If $a\star a=0$, then  $\br{a}=\cyclic{a}$ 
and moreover the multiplication in $\br{a}$ coincides with addition;
\item[ (ii)] If $a\in \zet{A}$, then 
$\cyclic{a}$ is an ideal of $A$;
\item[(iii)] If $a\in\zen{2}{A}$, then $\br{a}=\cyclic{a}+\cyclic{a\star a}$.
\end{enumerate}
\end{lemma}


\begin{proof}
(i) Since $a\star a=0$ it follows that $a^2=2a$. Using this fact and Lemma~\ref{l:l1} a simple induction argument allows us to conclude that $a^n=na$ for all natural numbers $n$. Furthermore, we have
\[
a=a^{-1}a^2=a^{-1}(a+a)=1+1-a^{-1}=-a^{-1}
\]
and a further straightforward induction implies that $a^{-n}=-na$ for all positive integers $n$.  Also for all integers $m,n$ we have $(ma)(na)=a^ma^n=a^{m+n}=(m+n)a$ so 
$\cyclic{a}$ is closed under multiplication. Additionally,
\[
(na)^{-1}=(a^n)^{-1}=a^{-n}= -na
\]
so 
$\cyclic{a}$ is also  a multiplicative group. Thus $\cyclic{a}$ is the subbrace generated by $a$ and the addition and multiplication on $\cyclic{a}$ coincide.

(ii) By (i), $\cyclic{a}$ is a subbrace of $A$. Moreover, if $b\in A$, then Lemma~\ref{l:l3} implies that $na\star b=b\star na=0\in\cyclic{a}$ for all integers $n$ so the result is immediate.

(iii) If $a\in \zen{2}{A}$, then $u=a\star a\in\zet{A}$ and hence $\cyclic{u}$ is an ideal of $A$ and $(a+\cyclic{u})\star (a+\cyclic{u})=0$. By (i) it follows that $\br{a+\cyclic{u}}=\cyclic{a+\cyclic{u}}=(\cyclic{a}+\cyclic{u})/\cyclic{u}$. Hence $\cyclic{a}+\cyclic{u}$ is a subbrace of $A$ containing $a$ so $\br{a}\leq \cyclic{a}+\cyclic{u}$. Since  it is clear that $a, u\in \br{a}$ we have $\br{a}=\cyclic{a}+\cyclic{u}$ as required.
\end{proof}

We shall require some knowledge of $\star$-hypercentral braces and the next few results incorporate the results we need. First we give an important property of the upper $\star$-hypercenter, which is analogous to the well-known result in group theory. This result appears as \cite[Proposition 2.10]{BJ21}. There is a similar result in \cite[Proposition 2.26]{CSV19} and \cite[Lemma 4.9]{BEFPT25}.

\begin{proposition}\label{p:prop4}
Let $A$ be a left brace and suppose that $K$ is a non-trivial ideal of $A$ such that $K\leq \zeta_{\infty}(\star,A)$. Then $K\cap \zeta(\star,A)$ is non-trivial.
\end{proposition}

Next we have (see \cite[Lemma 3.3]{DKS25}, for example).

 \begin{lemma}\label{l:hyperideal}
 Let $A$ be a left brace and let $S$ be a subbrace of $A$. For each ordinal $\al$, $S+\zen{\al}{A}$ is an ideal of $S+\zen{\al+1}{A}$.
 \end{lemma}

A brace $A$ is called \emph{locally Smoktunowicz-nilpotent} (or \emph{locally (centrally nilpotent)} or \emph{locally $\star$-nilpotent}) if every finitely generated subbrace of $A$ is Smoktunowicz-nilpotent. Such braces have been studied in \cite{BEFPT25} and \cite{DKS25}. Every $\star$-hypercentral brace is locally Smoktunowicz-nilpotent. Of particular interest here are the following two results 
which appear in \cite{DKS25}  and \cite{BEFPT25}. 
These results generalize results occurring in \cite{BJ21} and \cite{JVV22}. 


\begin{theorem}\label{t:cor49}
Let $A$ be a $\star$-hypercentral left brace. 
If the additive group  of $A$ is torsion-free, then the additive group of every factor of the upper $\star$-central series is torsion-free.
\end{theorem}

We  note that in \cite{BEFPT25} 
the authors define an element $a$ of a brace $A$ to be \emph{periodic} if $\br{a}$ is finite and in  this case the \emph{order of $a$} is $|\br{a}|$. If $\pi$ is a set of primes, then $a$ is a \emph{$\pi$-element} if its order is a $\pi$-number. A \emph{$\pi$-subbrace} is just a subbrace all of whose elements are  $\pi$-elements. If $p$ is a prime a \emph{Sylow $p$-subbrace} of $A$ is simply a maximal element of the set of all $p$-subbraces with respect to inclusion. The following catch-all theorem holds:

\begin{theorem}\label{t:marco}
Let $A$ be a locally Smoktunowicz-nilpotent left brace and let $T$ be the torsion-subgroup of the additive group of $A$.  Then 
\begin{enumerate}
\item[  (i)] $T$ coincides with the torsion subgroup of $(A,\cdot)$ and $T$ is an ideal of $A$;
\item[ (ii)] For every prime $p$, the Sylow $p$-subgroup $P$ of $T$ coincides with the Sylow $p$-subgroup of $(T,\cdot)$. Furthermore, $P$ is the unique maximal Sylow $p$-subbrace of $A$ and is an ideal of $A$.
\end{enumerate}
\end{theorem}

\section{Structure of torsion-free Smoktunowicz-nilpotent T-braces}\label{s:torfree}

We shall say that an element $a$ of a left brace $A$ has \emph{infinite additive order} if the element $a$ has infinite order as an element of $(A,+)$ and we say that $a$ has \emph{finite additive order $k$} if $a$ has order $k$ as an element of $(A,+)$. We shall also use at various points in our exposition the clear fact that  if $I$ is an ideal of a T-brace $A$, then $I$ is also a T-brace. Furthermore, if $A$ is a Smoktunowicz-nilpotent T-brace, then $A$ is a Dedekind brace. With this in mind the following result holds.

\begin{lemma}\label{l:l10}
Let $A$ be a T-brace and let $n$ be a natural number. Then  $\br{a}$ is an ideal of $A$ for all $a\in\zen{n}{A}$.
\end{lemma}

\begin{proof}
It follows from Lemma~\ref{l:hyperideal} that $\br{a}+\zen{k}{A}$ is an ideal of $\br{a}+\zen{k+1}{A}$ for all natural numbers $k$. From this and the fact that $A$ is a T-brace we easily deduce that $\br{a}$ is an ideal of $\br{a}+\zen{n}{A}$. However $\br{a}\leq \zen{n}{A}$ since $a\in\zen{n}{A}$. Hence $\br{a}$ is an ideal of $\zen{n}{A}$ which is itself an ideal of $A$. Thus $\br{a}$ is an ideal of $A$.
\end{proof}

\begin{lemma}\label{l:l5}
Let $A$ be a T-brace, let $n$ be a natural number and suppose that $a\in\zen{n}{A}\setminus\zet{A}$. If $a$ has infinite additive order, then $a\star a\neq 0$.
\end{lemma}

\begin{proof}
Suppose to the contrary that $a\star a=0$. Then Lemma~\ref{l:extra1} implies $\br{a}=\cyclic{a}$  and we deduce from Lemma~\ref{l:l10} that $\cyclic{a}$ is an ideal of $A$. 

Since $a\notin Z=\zet{A}$, there exists $u\in A$ such that either $u\star a\neq 0$ or $a\star u\neq 0$. Suppose first that $u\star a\neq 0$. We know $u\star a\in \cyclic{a}$ since $\cyclic{a}$ is an ideal of $A$. It follows from Proposition~\ref{p:prop4} 
that $\cyclic{a}\cap Z =\cyclic{ra} \neq 0$, for some positive integer $r$. However, Lemma~\ref{l:l3} implies $0=u\star ra=r(u\star a)$. Thus $u\star a$ has finite additive order in the group $\cyclic{a}$ which is a contradiction, since $u\star a\neq 0$. Thus $u\star a=0$ and consequently $a\star u\neq 0$.  However $a\star u\in \cyclic{a}$ since $\cyclic{a}$ is an ideal so $a\star u=sa=a^s$ for some integer $s$. Then using Proposition~\ref{p:p1} and a simple induction we have 
\begin{align*}
0&=ra\star u=a^r\star u=a^{r-1}\star (a\star u)+ (a\star u)+(a^{r-1}\star u)\\
&=a^{r-1}\star a^s+r(a\star u)=r(a\star u)
\end{align*}
and again $a\star u$ has finite order contrary to $\cyclic{a}$ having infinite order. The result follows.

\end{proof}

\begin{proposition}\label{p:p11}
Let $A$ be a left brace and let $a\in A$ be an element of infinite additive order. Let $c=a\star a$ and suppose that $a\star c=0$. Then 
$\cyclic{a}\cap\cyclic{c}=0$. In particular, if $a\in\zen{2}{A}$, then $\cyclic{a}\cap \cyclic{a\star a}=0$.
\end{proposition}

\begin{proof} 
If $c$ has finite additive order, then clearly $\cyclic{a}\cap \cyclic{c}=0$.
Thus, assume  that $c$ has infinite additive order and that $\cyclic{a}\cap \cyclic{c}\neq 0$. Then there are integers $m,n$ such that $ma=nc$. Then, using Proposition~\ref{p:p1}, 
\[
0=n(a\star c)=a\star nc=a\star ma=m(a\star a)=mc
\]
so that $c$ has finite additive order, a contradiction. Hence $\cyclic{a}\cap \cyclic{c}=0$ in this case also.
\end{proof}

The next result then follows easily.

\begin{corollary}\label{c:c7}
Let $A$ be a T-brace and suppose that the additive group of $\zet{A}$ is torsion-free. Then for every element $a\in\zen{2}{A}$ either $\br{a}=\cyclic{a}$ (and in this case $a\in\zet{A}$) or $\br{a}=\cyclic{a}\oplus\cyclic{a\star a}$.
\end{corollary}

\begin{proof} Let $a\in\zen{2}{A}$. Lemma~\ref{l:l5} shows that $a\star a=0$ if and only if $a\in\zet{A}$. The result follows by Lemma~\ref{l:extra1} and Proposition~\ref{p:p11}.
\end{proof}

We may generalize  Lemma~\ref{l:extra1} and Corollary~\ref{c:c7} as follows:

\begin{lemma}\label{l:extra2}
Let $A$ be a T-brace and suppose that $a\in\zen{n}{A}$ for some natural number $n$. Let $a_1=a, a_{i+1}=a_i\star a_i$ for all $i\geq 1$. Then 
\begin{enumerate}
\item [ (i)] $\br{a}$ is an ideal of $A$ and 
\[
\br{a}=\cyclic{a_1}+\cyclic{a_2}+\dots +\cyclic{a_n};
\]
\item [(ii)] If the additive group of $\zet{A}$ is torsion-free, then $\br{a}=\cyclic{a_1}\oplus \cyclic{a_2}\oplus\dots\oplus\cyclic{a_n}$.
\end{enumerate}
\end{lemma}

\begin{proof}
(i) We use induction on $n$. When $n=1$, then $a\star a=0$ and the result then follows by Lemma~\ref{l:extra1}. Suppose that the result is true for $n-1$ and let $a\in \zen{n}{A}$. Then $a_2=a\star a\in \zen{n-1}{A}$ and by induction we have that $B=\br{a_2}\norm \zen{n-1}{A}\norm A$ and $\br{a_2}=\cyclic{a_2}+\dots+ \cyclic{a_n}$. Now $A/B$ is a T-brace and $(a_1+B)\star (a_1+B)=a_2+B =B$ so by Lemma~\ref{l:extra1} we have 
\[
\br{a_1+B}=\cyclic{a_1+B}=(\cyclic{a_1}+B)/B\leq \zen{n}{A}/B.
\]
But $\zen{n}{A}/B$ is $\star$-nilpotent so every subbrace is a subideal. In particular this means that $\cyclic{a_1}+B$ is an ideal of $A$. However $a_1\in \cyclic{a_1}+B$ so $\br{a_1}\leq \cyclic{a_1}+B$. Moreover, $a_1, a_2,\dots, a_n\in \br{a}$ so $\br{a}=\cyclic{a_1}+\dots + \cyclic{a_n}$ as required.

(ii) We again use induction on $n$. If $a\in \zet{A}$, then the result follows immediately from (i). Assume that the result is true for $n-1$. By Theorem~\ref{t:cor49} the additive group of $\zen{n}{A}/\zen{n-1}{A}$ is torsion-free, so assuming that $a_1\notin\zen{n-1}{A}$ it follows that $a_1+\zen{n-1}{A}$ has infinite order and $\cyclic{a_1}\cap \zen{n-1}{A}=0$. However $\cyclic{a_2}+\dots+\cyclic{a_n}\subseteq \zen{n-1}{A}$ so
\[
\cyclic{a_1}\cap(\cyclic{a_2}+\dots +\cyclic{a_n})=0.
\]
Hence by (i) $\br{a}=\cyclic{a_1}\oplus  (\cyclic{a_2}+\dots +\cyclic{a_n}).$
However $a_2\in\zen{n-1}{A}$ and $\br{a_2}=\cyclic{a_2}+\dots +\cyclic{a_n}$ by (i) so by the induction hypothesis we have 
\[
\br{a_2}=\cyclic{a_2}\oplus\dots\oplus \cyclic{a_n}.
\]
Hence $\br{a}=\cyclic{a_1}\oplus \cyclic{a_2}\oplus\dots\oplus \cyclic{a_n}.$
Finally, if $a_1\in\zen{n-1}{A}$, then $a_2\in \zen{n-2}{A}$ and we can immediately apply the induction hypothesis to obtain the result.
\end{proof}

\begin{lemma}\label{l:l8}
Let $A$ be a left brace and let $a\in A$ be such that $a\star a=0$. Then the subbrace $\br{a, \zet{A}}=\cyclic{a}+\zet{A}$ is abelian.
\end{lemma}

\begin{proof}
Let $Z=\zet{A}$ and consider $C=\cyclic{a}+Z$ which is an additive subgroup of $A$. We first show that it is a subbrace and to this end let $c,d\in C$. Then there exist integers $m,n$ and elements $x,y\in Z$ such that $c= ma+x, d=na+y$. Since $x,y\in Z$,  Lemma~\ref{l:extra1} implies that $c=(ma)x=a^mx$ and $d=(na)y=a^ny$. Proposition~\ref{p:p1} then yields
\begin{align*}
c\star d &=(ma)x\star ((na)+y)=(ma)x\star na +(ma)x\star y\\
&=ma\star (x\star na)+x\star na+ma\star na
=n(a^m\star a)=0.
\end{align*}
Thus $\cyclic{a}+Z$ is a subbrace so $\br{a,Z}\leq \cyclic{a}+Z$. But clearly $\cyclic{a}+Z\leq \br{a,Z}$ so we have equality and the preceding computation shows it to be abelian.
\end{proof}


\begin{lemma}\label{l:l12}
Let $A$ be a left brace and let $a\in \zen{2}{A}$ be an element of infinite additive order. Let $c=a\star a$. If $c$ has infinite additive order,  
then $\br{2a}$ is not an ideal of $A$. 
\end{lemma}

\begin{proof}
Since $a\in\zen{2}{A}$ we have $a\star c=0$ and 
 Proposition~\ref{p:p11} implies that $\cyclic{a}\cap\cyclic{c}=0$.
Furthermore,  Lemma~\ref{l:extra1} enables us to deduce that $\br{a}=\cyclic{a}+\cyclic{c}=\cyclic{a}\oplus \cyclic{c}$. Let $w=2a\in\zen{2}{A}$ and set $u=w\star w\in\zet{A}$. It follows from Lemma~\ref{l:l3} that 
\[
u=2a\star 2a=4(a\star a)=4c.
\]
Furthermore, $w\star u=0=u\star w$. By Lemma~\ref{l:extra1} we have $\br{w}=\cyclic{w}+\cyclic{u}$ and since $\cyclic{a}\cap \cyclic{c}=0$ it follows that $\br{w}=\cyclic{w}\oplus \cyclic{u}$. Thus every element of $\br{w}$ can be written uniquely in the form $mw+nu=2ma+4nc$ for certain integers $m,n$. If $\br{w}$ is an ideal of $A$, then $a\star w=2c\in \br{w}$. However this cannot happen for any choice of the integers $m, n$ since $c$ has infinite additive order. It follows that $\br{w}$ is not an ideal of $A$.
\end{proof}

\begin{proof}[\textbf{Proof of Theorem~\ref{thma}}]
Suppose to the contrary, that $\zen{2}{A}\neq \zet{A}$ and choose $a\in\zen{2}{A}\setminus \zet{A}$. Theorem~\ref{t:cor49} implies that $a$ has infinite additive order. By Corollary~\ref{c:c7} we have $\br{a}=\cyclic{a}\oplus \cyclic{u}$ where $u=a\star a$. Let $c=2a\in\zen{2}{A}$. Then  Lemma~\ref{l:l10} implies that $\br{c}$ is an ideal of $A$, which contradicts Lemma~\ref{l:l12}. It follows that $\zen{\infty}{A}=\zet{A}$ as required.
\end{proof}

Corollary~\ref{cor1} is clear and Corollary~\ref{cor2} is immediate since Smoktunowicz-nilpotent braces are $\star$-nilpotent by \cite[Lemma 2.9]{DKS25}.

 Furthermore, Corollary~\ref{cor3} is also immediate since Theorem~\ref{t:marco} implies that $T$ is an ideal of $A$ and we can then apply Theorem~\ref{thma} to the factor brace $A/T$.

\section{Non-periodic Smoktunowicz-nilpotent T-braces}\label{s:nonper}

\begin{lemma}\label{l:l21}
Let $A$ be a T-brace. 
\begin{enumerate}
\item[ (i)] If $A$ is $\star$-hypercentral, then 
$A\star A$ is periodic;
\item[(ii)] if $\be$ is an ordinal and $a,b\in\zen{\be}{A}$, then $a\star b$ has finite additive order.
\end{enumerate}
\end{lemma}

\begin{proof}
(i) If $A$ is periodic, then it is clear that $A\star A$ is also periodic. If $A$ has an element of infinite additive order, then Corollary~\ref{cor3} implies that $A/T$ is abelian, where $T$ is the torsion subgroup of $(A,+)$. Hence $A\star A\leq T$.
%
%

(ii) The ideal $\zen{\be}{A}$ is a $\star$-hypercentral T-brace so the result then follows by (i).
\end{proof}

\begin{proposition}\label{p:p22}
Let $A$ be a T-brace. Then the additive group of $\zen{n}{A}/\zet{A}$ is periodic for all natural numbers $n$.
\end{proposition}

\begin{proof}
We first handle the case $n=2$. 
Let $Z=\zet{A}$ and suppose to the contrary that $a\in\zen{2}{A}$ is an element such that $a+Z$ has infinite additive order. 
Then it follows that $\cyclic{a}\cap Z=0$. Lemma~\ref{l:l21} shows that $u=a\star a$ has finite additive order $n$, say. Now $na\in\zen{2}{A}\setminus Z$ and hence Lemma~\ref{l:l5} implies that $na\star na\neq 0$. However $na\star na=n^2u=0$, using Lemma~\ref{l:l3},  a contradiction which proves the result.

More generally, if $n\in\mathbb{N}$ we note that $A/\zen{n-1}{A}$ is a $T$-brace so by the case above the additive group of $\zen{2}{A/\zen{n-1}{A}}/\zen{1}{A/\zen{n-1}{A}}$ is periodic. Thus
\[
\zen{n+1}{A}/\zen{n}{A}\cong (\zen{n+1}{A}/\zen{n-1}{A})/(\zen{n}{A}/\zen{n-1}{A})
\]
has periodic additive group and it follows easily that the additive group of $\zen{n}{A}/\zet{A}$ is likewise periodic for all $n\in\mathbb{N}$.
\end{proof}

\begin{lemma}\label{l:l28}
Let $A$ be a T-brace and let $a\in\zen{2}{A}$. If $a$ has infinite additive order, then $x\star a, a\star x\in \cyclic{a\star a}$ for all elements $x\in A$.
\end{lemma}

\begin{proof}
Lemma~\ref{l:extra1} implies that $B=\br{a}=\cyclic{a}+\cyclic{c}$, where $c=a\star a$. Let $C=\cyclic{a\star a}\leq \zet{A}$. By Lemma~\ref{l:l21}  $c$ has finite additive order and hence $B=\cyclic{a}\oplus C$. Since $C\leq\zet{A}$, $C$ is an ideal of $A$ and in $A/C$ we have $(a+C)\star (a+C)=C$. Since $a+C\in\zen{2}{A/C}$ has infinite additive order, Lemma~\ref{l:l5} implies that $a+C\in\zet{A/C}$. From this it follows that $(x+C)\star (a+C)=(a+C)\star (x+C)=C$ for all $x\in A$ and hence $x\star a, a\star x\in C$ as required.
\end{proof}

\begin{lemma}\label{l:l29}
Let $A$ be a T-brace and let $a\in\zen{2}{A}$. If $a$ has finite additive order and the additive group of $\zen{2}{A}$ is not periodic, then $x\star a, a\star x\in \cyclic{a\star a}$ for all elements $x\in A$.
\end{lemma}

\begin{proof}
By Proposition~\ref{p:p22} the additive group of $\zet{A}$ is not periodic and we let $d\in\zet{A}$ be an element of infinite additive order. It is clear that $d\star a=a\star d=d\star d=0$ and hence 
\[
(d+a)\star (d+a)=a\star a,
\]
using Proposition~\ref{p:p1}. 
Since $d+a\in \zen{2}{A}$ has infinite additive order, Lemma~\ref{l:l28} implies that 
\[
x\star (d+a)=x\star a\in \cyclic{(d+a)\star (d+a)}=\cyclic{a\star a}
\]
and likewise $a\star x\in \cyclic{a\star a}$. This completes the proof.
\end{proof}

\begin{lemma}\label{c:c212}
Let $A$ be a T-brace and let $a\in\zen{2}{A}$ be an element of infinite additive order. Suppose that the torsion subgroup of the additive group of $\zet{A}$ is a $p$-group for some prime $p$. Then $a\star a$ has additive order at most $p$.
\end{lemma}

\begin{proof}
We know from  Lemma~\ref{l:l21} that $c=a\star a$ has finite order $p^k$ for some natural number $k$. If $k$ is even let $l=k/2$ and if $k$ is odd let $l=(k+1)/2$ and set $w=p^la$. Using Lemma~\ref{l:l3} we have $w\star w=p^kc $ if $k$ is even and $w\star w=p^{k+1}c$ if $k$ is odd and in either case we obtain $w\star w=0$. Of course, $w$ has infinite additive order, so Lemma~\ref{l:l5} implies that $w\in\zet{A}$. Hence $\br{w}$ is an ideal of $A$ and in particular we have
\[
0=a\star w=p^l(a\star a)=p^lc.
\]
This implies that $k$ is  odd and indeed in this case $(k+1)/2\geq k$ so $k=1$. The result follows.
\end{proof}

\begin{corollary}\label{c:c213}
Let $A$ be a T-brace and let $a\in\zen{2}{A}$ be an element of finite additive order. Suppose that the torsion subgroup of the additive group of $\zet{A}$ is a $p$-group for some prime $p$. If the additive group of $\zen{2}{A}$ is not periodic, then the additive order of $a\star a$ is at most $p$.
\end{corollary}

\begin{proof}
Proposition~\ref{p:p22} implies that the additive group of $\zet{A}$ is not periodic and we let $d\in\zet{A}$ be an element of infinite additive order. Then $d+a$ also has infinite additive order and Lemma~\ref{c:c212} implies that  $p((d+a)\star (d+a))=0$. Since $d\star a=a\star d=d\star d=0$ it follows that 
\[
p(a\star a)= p((a+d)\star (a+d))=0,
\]
giving the result.
\end{proof}

\begin{lemma}\label{l:l210}
Let $A$ be a T-brace and suppose that the torsion subgroup of the additive group of $\zet{A}$ is a $p$-group for some prime $p$. Suppose that the additive group of $\zen{2}{A}$ is not periodic. If $a,b\in\zen{2}{A}$ and $a\star a\neq 0\neq b\star b$, then $\cyclic{a\star a}\cap\cyclic{b\star b}\neq 0$. In this case $\cyclic{a\star a}=\cyclic{b\star b}$.
\end{lemma}

\begin{proof}
Using Lemma~\ref{l:l21} we see that $u=a\star a, v=b\star b$ both have finite additive order. Suppose, for a contradiction, that $\cyclic{u}\cap\cyclic{v}=0$. It follows from Lemmas~\ref{l:l28} and \ref{l:l29} that $a\star b, b\star a\in \cyclic{u}\cap\cyclic{v}=0$. Thus $a+b =ab$ and  Proposition~\ref{p:p1} gives
\begin{align*}
(a+b)\star (a+b)&=(ab)\star (a+ b) =(ab)\star a +(ab)\star b\\
&=a\star a+ b\star b=u+v.
\end{align*}
 Again using Lemmas~\ref{l:l28} and \ref{l:l29}  we have $a\star (a+b)\in \cyclic{(a+b)\star (a+b)}=\cyclic{u+v}$. On the other hand $a\star (a+b)=a\star a+ a\star b=u$. Thus $u\in\cyclic{u+v}$ and a contradiction  follows easily, since $\cyclic{u}\cap\cyclic{v}=0$ and $\cyclic{u}, \cyclic{v}$ are $p$-groups. Thus $\cyclic{u}\cap\cyclic{v}\neq 0$ and Lemma~\ref{c:c212} and Corollary~\ref{c:c213} prove the rest of the result.
\end{proof}

\begin{lemma}\label{l:l214}
Let $A$ be a T-brace and suppose that the torsion subgroup of the additive group of $\zet{A}$ is a $p$-group for some prime $p$. If $\zen{2}{A}\neq \zet{A}$ and $K=\zen{2}{A}\star \zen{2}{A}$,  then $\zen{2}{A}/K\leq \zet{A/K}$. Furthermore, $K$ has order $p$.
\end{lemma}

\begin{proof}
Since a nonperiodic abelian group is generated by its elements of infinite order, it is easy to see that there is an element $a$ of infinite additive order such that $a\in\zen{2}{A}\setminus\zet{A}$.  Lemma~\ref{l:l5} and Corollary~\ref{c:c213} imply that $c=a\star a\neq 0$ has additive order $p$. Let $L=\cyclic{c}\leq \zet{A}$. Then $L$ is an  ideal of $A$.  If $b\in\zen{2}{A}$ is such that $b\star b\neq 0$, then 
Lemma~\ref{l:l210} shows that $L= \cyclic{b\star b}$. Lemma~\ref{l:l28} and Lemma~\ref{l:l29} imply that if $x\in A, y\in\zen{2}{A}$, then $x\star y, y\star x\in L$. It follows that $\zen{2}{A}/L\leq \zet{A/L}$ and $K=\zen{2}{A}\star \zen{2}{A}\leq L$, as required.
\end{proof}

\begin{corollary}\label{c:c215}
Let $A$ be a T-brace and suppose that the torsion subgroup of the additive group of $\zet{A}$ is a $p$-group for some prime $p$. Suppose that $A$ is a Smoktunowicz-nilpotent brace such that $\zl{A}=k$ for  some positive integer $k$. If the additive group of $A$ is not periodic, then $K=A\star A$ is finite of order at most $p^{k-1}$.
\end{corollary}

\begin{proof}
We use induction on $\zl{A}$. When $k=1$, then $A$ is an abelian brace and $A\star A=0$ so the result easily follows in this case.  Suppose that $k>1$ and the result is true for all Smoktunowicz-nilpotent braces $B$ such that $(B,+)$ is not periodic and $\zl{B}\leq k-1$. In this case $\zen{2}{A}\neq \zet{A}$ so $L=\zen{2}{A}\star \zen{2}{A}$ has order $p$ and $\zen{2}{A}/L\leq \zet{A/L}$ by Lemma~\ref{l:l214}. From this it follows that $\zl{A/L}<\zl{A}$ and hence $K/L=(A/L)\star (A/L)$ has order at most $p^{k-2}$ by the induction hypothesis. Thus $K=A\star A$ has order at most $p^{k-1}$, as required.
\end{proof}

\begin{proof}[\textbf{Proof of Theorem~\ref{thmb}}]
Let $T$ be the torsion subgroup of the additive group of $A$. By Theorem~\ref{t:marco} $T$ is an ideal of $A$ and for every prime $p$ the Sylow $p$-subgroup $T_p$ of the additive group of $A$ is also an ideal of $A$, as Theorem~\ref{t:marco} shows. Let  $Q_p=\dir{q}{\Pi(A), q\neq p}{T}$.  The torsion subgroup of the additive group of the brace $A/Q_p$ is a $p$-group. If $A/Q_p$ is abelian, then the Sylow $p$-subgroup of $A\star A$ is trivial. If $A/Q_p$ is not an abelian brace we use Corollary~
\ref{c:c215} to deduce that $(A/Q_p)\star (A/Q_p)=(A\star A)/Q_p$ is a finite $p$-subgroup of order at most $p^{k-1}$. Since $Q_p$ is a $p'$-group it follows that the Sylow $p$-subgroup of the additive group of $A\star A$ has order at most $p^{k-1}$ as required.
\end{proof}


\begin{thebibliography}{10}

\bibitem{BEFPT25}
A.~Ballester-Bolinches, R.~Esteban-Romero, M.~Ferrara, V.~P\'{e}rez-Calabuig,
  and M.~Trombetti, \emph{Central nilpotency of left skew braces and solutions
  of the {Y}ang-{B}axter equation}, Pacific J. Math. \textbf{335} (2025),
  no.~1, 1--32.

\bibitem{BEKP25}
A.~Ballester-Bolinches, R.~Esteban-Romero, L.~A. Kurdachenko, and
  V.~P\'{e}rez-Calabuig, \emph{On left braces in which every subbrace is an
  ideal}, Results Math. \textbf{80} (2025), no.~1, Paper No. 21, 21.

\bibitem{BJ21}
M.~Bonatto and P.~Jedli\v{c}ka, \emph{Central nilpotency of skew braces}, J.
  Algebra Appl. \textbf{22} (2023), no.~12, Paper No. 2350255, 16.

\bibitem{fC18}
F.~Ced\'{o}, \emph{Left braces: solutions of the {Y}ang-{B}axter equation},
  Adv. Group Theory Appl. \textbf{5} (2018), 33--90.

\bibitem{CSV19}
F.~Ced\'{o}, A.~Smoktunowicz, and L.~Vendramin, \emph{Skew left braces of
  nilpotent type}, Proc. Lond. Math. Soc. (3) \textbf{118} (2019), no.~6,
  1367--1392.

\bibitem{CFT23}
I.~Colazzo, M.~Ferrara, and M.~Trombetti, \emph{On derived-indecomposable
  solutions of the {Y}ang-{B}axter equation}, to appear, Publicacions
  Matem\`{a}tiques.

\bibitem{DKS25}
M.~R. Dixon, L.~A. Kurdachenko, and I.~Ya. Subbotin, \emph{Generalized
  nilpotent braces and nilpotent groups}, Int. J. Group Theory \textbf{14}
  (2025), no.~1, 25--45.

\bibitem{wG57}
W.~Gasch\"{u}tz, \emph{Gruppen, in denen das {N}ormalteilersein transitiv ist},
  J. Reine Angew. Math. \textbf{198} (1957), 87--92.

\bibitem{JVV22}
E.~Jespers, A.~van Antwerpen, and L.~Vendramin, \emph{Nilpotency of skew braces
  and multipermutation solution of the {Y}ang-{B}axter equation}, Commun.
  Contemp. Math. \textbf{25} (2023), no.~9, Paper No. 2250064, 20.

\bibitem{KLMO23}
L.~A. Kurdachenko, P.~Longobardi, M.~Maj, and J.~Otal, \emph{On pairs of
  antagonistic subgroups and theirs influence on the structure of groups}, Int.
  J. Group Theory \textbf{12} (2023), no.~2, 81--98.

\bibitem{KSY18}
L.~A. Kurdachenko, I.~Ya. Subbotin, and V.~S. Yashchuk, \emph{Leibniz algebras
  whose subideals are ideals}, J. Algebra Appl. \textbf{17} (2018), no.~8,
  1850151, 15.

\bibitem{dR64}
D.~J.~S. Robinson, \emph{Groups in which normality is a transitive relation},
  Proc. Cambridge Philos. Soc. \textbf{60} (1964), 21--38.

\bibitem{wR05}
W.~Rump, \emph{A decomposition theorem for square-free unitary solutions of the
  quantum {Y}ang-{B}axter equation}, Adv. Math. \textbf{193} (2005), no.~1,
  40--55.

\bibitem{aS18}
A.~Smoktunowicz, \emph{On {E}ngel groups, nilpotent groups, rings, braces and
  the {Y}ang-{B}axter equation}, Trans. Amer. Math. Soc. \textbf{370} (2018),
  no.~9, 6535--6564.

\bibitem{aS22}
\bysame, \emph{Algebraic approach to {R}ump's results on relations between
  braces and pre-{L}ie algebras}, J. Algebra Appl. \textbf{21} (2022), no.~3,
  Paper No. 2250054, 13.

\bibitem{mT23}
M.~Trombetti, \emph{The structure skew brace associated with a finite
  non-degenerate solution of the {Y}ang-{B}axter equation is finitely
  presented}, Proc. Amer. Math. Soc. \textbf{152} (2024), no.~2, 573--583.

\end{thebibliography}

\providecommand{\bysame}{\leavevmode\hbox to3em{\hrulefill}\thinspace}
\providecommand{\MR}{\relax\ifhmode\unskip\space\fi MR }
\providecommand{\MRhref}[2]{%
  \href{http://www.ams.org/mathscinet-getitem?mr=#1}{#2}
}
\providecommand{\href}[2]{#2}

\end{document}